\documentclass[journal]{IEEEtran}

\IEEEoverridecommandlockouts                              

\usepackage{changes}
\usepackage{authblk}
\usepackage{amsfonts}
\usepackage{amsthm}
\usepackage{amsmath,amssymb}       
\usepackage{hyperref,cite}
\usepackage{cleveref}
\usepackage{color}
\usepackage{algorithm}
\usepackage{algorithmic}
\let\labelindent\relax
\usepackage{enumitem}
\usepackage{float}
\usepackage{graphicx}
\usepackage{booktabs}
\usepackage{threeparttable}
\usepackage{hhline}
\newtheorem{theorem}{Theorem}

\newtheorem{remark}{Remark}
\newtheorem{lemma}{Lemma}
\DeclareMathOperator*{\argmin}{arg\,min} 
\newcommand{\diag}{\mathop{\rm diag}\nolimits}

\newcommand{\ba}[1]{\begin{array}{#1}}
\newcommand{\ea}{\end{array}}

\begin{document}
\title{\LARGE \bf
A direct optimization algorithm for input-constrained MPC}
\author{Liang Wu$^{1}$, Richard D. Braatz$^{1}$, Fellow, IEEE%
\thanks{$^{1}$Massachusetts Institute of Technology, Cambridge, MA 02139, USA, {\tt\small \{liangwu,braatz\}@mit.edu}\\
}
}

\maketitle

\begin{abstract}
Providing an execution time certificate is a pressing requirement when deploying Model Predictive Control (MPC) in real-time embedded systems such as microcontrollers. Real-time MPC requires that its worst-case (maximum) execution time must be theoretically guaranteed to be smaller than the sampling time in closed-loop. This technical note considers input-constrained MPC problems and exploits the structure of the resulting box-constrained QPs. Then, we propose a \textit{cost-free} and \textit{data-independent} initialization strategy, which enables us, for the first time, to remove the initialization assumption of feasible full-Newton interior-point algorithms. We prove that the number of iterations of our proposed algorithm is \textit{only dimension-dependent} (\textit{data-independent}), \textit{simple-calculated}, and \textit{exact} (not \textit{worst-case}) with the value $\left\lceil\frac{\log(\frac{2n}{\epsilon})}{-2\log(\frac{\sqrt{2n}}{\sqrt{2n}+\sqrt{2}-1})}\right\rceil \!+ 1$, where $n$ denotes the problem dimension and $\epsilon$ denotes the constant stopping tolerance. These features enable our algorithm to trivially certify the execution time of nonlinear MPC (via online linearized schemes) or adaptive MPC problems. The execution-time-certified capability of our algorithm is theoretically and numerically validated through an open-loop unstable AFTI-16 example.
\end{abstract}


\section{Introduction}
Model predictive control (MPC) generally requires solving an online quadratic programming (QP) problem at each sampling time in a real-time closed-loop. Deploying real-time MPC on embedded production platforms such as microcontrollers, has to meet a key requirement, called an execution time certificate, which is that the execution time of the adopted QP algorithm must be theoretically guaranteed to be less than the given sampling time. 

This execution time certificate has garnered increasing attention within recent years and is still an active research area \cite{richter2011computational, giselsson2012execution, patrinos2013accelerated, cimini2017exact, arnstrom2019exact, cimini2019complexity,arnstrom2020complexity, arnstrom2021unifying, okawa2021linear}. All these works are based on the assumption that \textit{the adopted computation platform performs a fixed number of floating-point operations ([flops]) in constant time},
$$
\textrm{execution time} = \frac{\textrm{total [flops] required by the algorithm}}{\textrm{average [flops] processed per second}}~[s].
$$
where one flop is defined to be one multiplication, subtraction, addition, or division of two floating-point numbers, then, the execution time can be derived by analyzing the total worst-case [flops] which is equivalent to analyzing the worst-case number of iterations if each iteration takes invariant [flops]. This technical note also follows this assumption to certify the execution time of input-constrained MPC problems by analyzing the number of iterations.

Currently, most MPC algorithms obtain the \textit{average} and \textit{worst-case} execution time by performing thousands of closed-loop experiments and then statistically analyzing all execution times. This approach is heuristic and lacks a theoretical certificate. More importantly, this heuristic approach cannot build an explicit and exact relationship between the execution time and the MPC settings such as the length of prediction horizons. Consequently, choosing the appropriate sampling time and embedded processor type for a given MPC setting relies on heavy calibration work \cite{forgione2020efficient}. 

Furthermore, most existing works \cite{wang2009fast, ferreau2014qpoases,stellato2020osqp,wu2023simple, wu2023construction} use the statistical \textit{average} execution time to claim that their proposed algorithm is fast, but in fact, the \textit{worst-case execution} time matters a lot more than \textit{average} execution time. This is not only because the certified \textit{worst-case} execution time (not the \textit{average} execution time), is used to choose the sampling time, but also because only when the \textit{worst-case} execution time (not the \textit{average}  execution time) is small, MPC can be applied to fast dynamic systems.

In MPC applications, online optimization problems have time-changing problem data (but the problem dimension is time-invariant) in closed-loop as the feedback state, the set-point reference signal, or even the model, are all time-changing (such as in online linearized nonlinear MPC and adaptive MPC). Simply put, the number of iterations depends on the convergence speed of an optimization algorithm and the distance between the initial point and the optimal point, both of which are dependent on the data of optimization problems. Therefore, time-changing problem data poses a big challenge in obtaining the execution time certificate of real-time MPC. This paper aims to develop an optimization algorithm that has \textit{only dimension-dependent} (\textit{data-independent}), \textit{simple-calculated}, and \textit{exact} number of iterations, then enabling the certification of the execution time of linear MPC, nonlinear MPC (via online linearized scheme), and adaptive MPC problems.

\subsection{Related work}
In \cite{richter2011computational}, the input-constrained linear MPC problem, resulting in a box-constrained QP (Box-QP) problem, is considered and solved by Nesterov's fast gradient method. They derived a conservative iteration complexity bound which is not only very computationally complicated but also dependent on the problem data. In \cite{giselsson2012execution, patrinos2013accelerated}, the general linear MPC problem with input and state constraints is considered and then transformed into the dual problems, which are solved by the accelerated gradient projection methods. And their worst-case iteration complexity is also dependent on the problem data like the Hessian matrix of the dual problem.  Their \textit{data-dependent} iteration complexity result cannot guarantee the time-invariant number of iterations in online linearization-based MPC problems, such as Real-Time Iteration (RTI ) based nonlinear MPC \cite{gros2020linear}).

In addition to the above first-order method, another popular class of optimization methods is the active-set method. Although active-set methods often run fast in small/medium-scale problems but could have an exponential number of iterations in the worst-case \cite{klee1972good}. In \cite{cimini2017exact, arnstrom2019exact, cimini2019complexity,arnstrom2020complexity, arnstrom2021unifying}, they use active-set methods to solve general linear MPC problems and provide the certification procedure of the worst-case number of iterations. However, their certification procedure relies on the computationally complicated and expensive (thus offline) worst-case partial enumeration technique. This also happens in the work \cite{okawa2021linear} which only considers input-constrained MPC problems. The authors in \cite{okawa2021linear} proposed an $N$-step algorithm (the worst-case iteration complexity is the problem dimension $N$) if the \textit{modified N-step vector} is given. However, the \textit{modified N-step vector} is found by solving a linear programming problem. 

To summarize, all current execution time certificate works of MPC are 
either \textit{data-dependent} (like the first-order methods \cite{richter2011computational, giselsson2012execution, patrinos2013accelerated}), or rely on computationally 
complicated and expensive techniques (like the active-set methods \cite{cimini2017exact, arnstrom2019exact, cimini2019complexity,arnstrom2020complexity, arnstrom2021unifying, okawa2021linear}), making them unsuitable for nonlinear MPC (via online linearized schemes) and adaptive MPC problems. Therefore, to the best of the authors' knowledge, no works extend these algorithms\cite{richter2011computational, giselsson2012execution, patrinos2013accelerated, cimini2017exact, arnstrom2019exact, cimini2019complexity,arnstrom2020complexity, arnstrom2021unifying, okawa2021linear} to certify the execution time of nonlinear MPC problems.

To address the applicability limits of the above first-order and active-set methods, this paper turns to interior-point methods (IPM). IPMs have been exploited in MPC applications, as seen in works such as \cite{rao1998application,zanelli2020forces}, which primarily focus on how to improve the \textit{average} computational efficiency in practice, but lacks execution time certificate in theory. For example, the Mehrotra predictor-corrector IPMs \cite{mehrotra1992implementation}) have been the basis for most interior-point software such as \cite{zanelli2020forces} due to their practical fast convergence speed. However, Mehrotra predictor-corrector IPMs are heuristic and may diverge on some examples \cite[see page 411]{nocedal2006numerical}, \cite{cartis2009some}, without theoretical global convergence proof.

IPMs are well-known for their theoretically certified polynomial time complexity \cite{wright1997primal}, but no work has adopted IPM to certify the execution time of real-time MPC problems because of the ``irony of IPM", a puzzling gap between practical and theoretical computational efficiency \cite{renegar2001mathematical}. In practice, heuristic Mehrotra predictor-corrector IPMs often take less than $<50$ iterations 
(behaving $O(\log(n))$ iteration complexity). 
In theory, the best worst-case iteration complexity of IPMs is $O(\sqrt{n})$, where $n$ denotes the problem dimension. Specifically, according to whether the initial point is strictly feasible or not, IPMs can be divided into infeasible IPMs and feasible IPMs, and the best worst-case iteration complexity of infeasible IPMs and feasible IPMs are certified $O(n)$ and $O(\sqrt{n})$ \cite{wright1997primal,ye2011interior}, respectively. Feasible IPMs are preferable to infeasible IPMs in providing a faster execution time certificate for real-time MPC problems.

However, feasible IPMs are rarely used in practical applications, let alone MPC applications because they have an unrealistic assumption that the initial point is strictly feasible and located in a narrow-centered neighborhood. For example, existing works on feasible IPMs for general LPs \cite{darvay2002weighted}, convex QPs \cite{achache2006new,boudjellal2023new}, and monotone linear complementarity problems \cite{achache2015full} all rely on this assumption. Removing this assumption, namely finding this specified initial point for feasible IPMs, requires solving another linear program (LP), which not only significantly increases the computational cost, but more importantly introduces another challenge: certifying the execution time of the LP itself.  Therefore, feasible IPMs are never used to certify the execution time of real-time MPC.

\subsection{Contribution}
This technical note for the first time develops a tailored and practical feasible IPM algorithm to certify the execution time of input-constrained MPC problems, to enjoy the current best theoretical $O(\sqrt{n})$ iteration complexity. Our novel contributions are four-fold:
\begin{itemize}
    \item[1)] \textit{Removes the assumption} of previous works that the initial point is strictly feasible and located in a narrow neighborhood of the central path. By only considering input-constrained MPC problems, we then exploit the structure of the resulting box-constrained QP (Box-QP) and for the first time innovatively propose a \textit{cost-free} and \textit{data-independent} initialization strategy.
    \item[2)] \textit{Very simple to implement} our proposed feasible IPM algorithm, which adopts full-Newton step thus without a line search procedure.
    \item[3)] \textit{Only dimension-dependent} (data-independent), \textit{simple-calculated}, and \textit{exact} (not worst-case or maximum) number of iterations,
    \[\left\lceil\frac{\log(\frac{2n}{\epsilon})}{-2\log(\frac{\sqrt{2n}}{\sqrt{2n}+\sqrt{2}-1})}\right\rceil + 1,\]is achieved by our proposed algorithm, where $n$ denotes the problem dimension of Box-QP and $\epsilon$ denotes the constant specifying the stopping accuracy (e.g., $1\times 10^{-6}$). Thanks to being \textit{only dimension-dependent} and having \textit{exact} computational complexity, our optimization algorithm is \textit{direct}, in the same manner as direct methods for solving linear equations $Ax=b$ (Cholesky decomposition or QR decomposition) which also have \textit{only dimension-dependent} (data-independent) and \textit{exact} computational complexity. To the best of the author's knowledge, a \textit{direct} optimization algorithm is reported for the first time.
    \item[4)] A \textit{simple-calculated} execution-time certificate is provided by our proposed algorithm.
\end{itemize}

\subsection{Notation}
$\mathbb{R}^n$ denotes the space of $n$-dimensional real vectors, $\mathbb{R}^n_{++}$ is the set of all positive vectors of $\mathbb{R}^n$, and $\mathbb{N}_+$ is the set of positive integers. For a vector $z\in\mathbb{R}^n$, $\|z\|=\sqrt{z_1^2+z_2^2+\cdots+z_n^2}$, $\|z\|_1=\sum_{i=1}^{n}|z_i|$, $\|z\|_{\infty}=\max_i |z_i|$, $\diag(z):\mathbb{R}^n\rightarrow\mathbb{R}^{n\times n}$ maps a vector $z$ to its corresponding diagonal matrix, and $z^2 = (z_1^2,z_2^2,\cdots,z_n^2)^\top$. Given two vectors $z,y \in\mathbb{R}^n_{++}$, their Hadamard product is $zy = (z_1y_1,z_2y_2,\cdots{},z_ny_n)^\top$, $\big(\frac{z}{y}\big)=\big(\frac{z_1}{y_1},\frac{z_2}{y_2},\cdots{},\frac{z_n}{y_n}
\big)^{\!\top}$,  $\sqrt{z}=\left(\sqrt{z_1},\sqrt{z_2},\cdots{},\sqrt{z_n}\right)^{\!\top}$ and $z^{-1}=(z_1^{-1},z_2^{-1},\cdots{},z_n^{-1})^{\top}$. The vector of all ones is denoted by $e=(1,\cdots{},1)^\top$. $\left\lceil x\right\rceil$ maps $x$ to the least integer greater than or equal to $x$. For $z,y\in\mathbb{R}^n$, let $\mathrm{col}(z,y)=[z^{\top},y^{\top}]^{\top}$.

\section{Input-constrained MPC}
In a closed-loop input-constrained MPC setting, at each sampling time $t$, a parametric Box-QP,
\begin{equation}\label{problem_input_MPC}
    \begin{aligned}
        \min_y&\quad\frac{1}{2}y^\top Q(t) y + y^\top d(t)\\\
        \mathrm{s.t.}&\quad  l(t) \leq y \leq u(t)
    \end{aligned}    
\end{equation}
needs to be solved within each sampling interval where $y \in \mathbb{R}^{n}$ denotes the optimization variables. The problem data, including the symmetric positive definite $Q(t)\in \mathbb{R}^{n \times n}$, the vector $d(t)\in \mathbb{R}^n$, and the lower and upper bounds $l(t), u(t)\in\mathbb{R}^n$ (bounded and $l(t) < u(t)$), may be time-varying. 

We assume that the above Box-QP formulation guarantees the stability of input-constrained MPC by choosing the positive definite terminal penalty matrix and the horizon length appropriately, ensuring the fulfillment of sufficient stability criteria, as discussed in \cite{richter2011computational}.

Applying the coordinate transformation,
\[
z = 2 \diag(u(t)-l(t))^{-1} y - 2 \diag(u(t)-l(t))^{-1}l - e
\]
results in an equivalent Box-QP with scaled box constraints,
\begin{subequations}\label{problem_Box_QP}
    \begin{align}
        z^*=&\argmin_z \tfrac{1}{2}z^\top H z + z^\top h\label{problem_Box_QP_objective}\\
        &\ \mathrm{s.t.}\  -e \leq z \leq e\label{proble_Box_QP_box_constraint}
    \end{align}
\end{subequations}
where $H = \diag(u(t)-l(t))Q(t)\diag(u(t)-l(t))$ and $h = \diag(u(t)-l(t))Q(t)(u(t)-l(t)+2d(t))$. The optimal solution $y^*$ of the original input-constrained MPC (\ref{problem_input_MPC}) can be recovered by
\begin{equation}\label{eqn_z*_y*}
    y^* = \tfrac{1}{2}\diag(u(t)-l(t))z^*+\tfrac{1}{2}(u(t)+l(t))
\end{equation}
In the below sections, the scaled box-QP (\ref{problem_Box_QP}) is used to derive and analyze the proposed algorithm. 

\section{Feasible full-Newton IPM}
According to \cite[Ch 5]{boyd2004convex}, the Karush–Kuhn–Tucker (KKT) condition of the scaled box-QP \eqref{problem_Box_QP} is the nonlinear equations,
\begin{subequations}\label{eqn_KKT}
\begin{align}
    Hz + h + \gamma - \theta = 0,\label{eqn_KKT_a}\\
    z + \phi - e=0,\label{eqn_KKT_b}\\
    z - \psi + e=0,\label{eqn_KKT_c}\\
    \gamma \phi = 0,\label{eqn_KKT_d}\\
    \theta \psi = 0,\label{eqn_KKT_e}\\
    (\gamma,\theta,\phi,\psi)\geq0.\label{eqn_KKT_f}
\end{align}
\end{subequations}
where $\gamma$ and $\theta$ denote the Lagrangian variable of the lower bound and upper bound, respectively, and $\phi$ and $\psi$ denote the slack variable of the lower bound and upper bound, respectively.

Primal-dual IPMs use Newton's method to determine the search direction for solving these nonlinear equations. Newton's method involves linearizing equation \eqref{eqn_KKT} around the current iterate and then solving a system of linear equations that results from this process. The solution to these equations provides the search direction $(\Delta z,\Delta\gamma,\Delta\theta,\Delta\phi,\Delta\psi)$. A full-Newton step often violates the bound $(\gamma,\theta,\phi,\psi)\geq0$
so the next iterate $(z,\gamma,\theta,\phi,\psi)+\alpha(\Delta z,\Delta\gamma,\Delta\theta,\Delta\phi,\Delta\psi)$ with a line search parameter $\alpha \in (0,1]$ is used to ensure not exceeding the bound. The straightforward application of Newton's method frequently results in small steps ($\alpha \ll 1$) before reaching the limit, which hinders significant progress toward finding a solution. 

A popular IPM is the path-following approach. This involves introducing a positive parameter $\tau$ to replace \eqref{eqn_KKT_d} and \eqref{eqn_KKT_e} with the equations,
\begin{subequations}\label{eqn_KKT_tau}
\begin{align}
    \gamma \phi = \tau^2 e,\label{eqn_KKT_tau_d}\\
    \theta \psi = \tau^2 e.\label{eqn_KKT_tau_e}
\end{align}
\end{subequations}
It has been shown that there exists one unique solution $(z_{\tau},\gamma_{\tau},\theta_{\tau},\phi_{\tau},\psi_{\tau})$ and the path $\tau \rightarrow (z_{\tau},\gamma_{\tau},\theta_{\tau},\phi_{\tau},\psi_{\tau})$ is called the {\em central path} \cite{wright1997primal}. As $\tau$ approaching $0$, $(z_{\tau},\gamma_{\tau},\theta_{\tau},\phi_{\tau},\psi_{\tau})$ goes to a solution of \eqref{eqn_KKT}. Simply put, primal-dual path-following methods apply Newton's method to \eqref{eqn_KKT_tau} and explicitly restrict the iterates to a neighborhood of the central path, which is an arc of strictly feasible points. The primal-dual feasible set $\mathcal{F}$ and strictly feasible set $\mathcal{F}^0$ are defined as
\[
\begin{aligned}
    \mathcal{F}=\{(z,\gamma,\theta,\phi,\psi)|\eqref{eqn_KKT_a}\mathrm{-}\eqref{eqn_KKT_c},(\gamma,\theta,\phi,\psi)\geq0\},\\
    \mathcal{F}^0=\{(z,\gamma,\theta,\phi,\psi)|\eqref{eqn_KKT_a}\mathrm{-}\eqref{eqn_KKT_c},(\gamma,\theta,\phi,\psi)>0\}.
\end{aligned}
\]
Thanks to the structure of the KKT condition (\ref{eqn_KKT}), we are able to propose a novel cost-free initialization strategy for feasible path-following IPM with application to input-constrained MPC.

\subsection{Strictly feasible initial point}
To obtain the best result of theoretical worst-case complexity, a good strictly feasible initial point is necessary.
\begin{remark}\label{remark_initialization_stragegy}
For $h=0$, the optimal solution of Problem \ref{problem_Box_QP} is $z^*=0$. For $h\neq0$,  first scale the objective \eqref{problem_Box_QP_objective} as 
\[
\min_z \tfrac{1}{2} z^\top (\tfrac{2\lambda}{\|h\|_\infty}H) z + z^\top (\tfrac{2\lambda}{\|h\|_\infty}h)
\]
which does not affect the optimal solution and the scalar $\lambda \in (0,1)$. Denoting $\tilde{H} = \frac{1}{\|h\|_\infty}H$ and $\tilde{h}=\frac{1}{\|h\|_\infty} h$, we have that $\|\tilde{h}\|_\infty =1$.
Then \eqref{eqn_KKT_a} is replaced by 
\[
2\lambda \tilde{H}z+2\lambda\tilde{h}+\gamma-\theta=0
\]
the initialization strategy to solve Problem \ref{problem_Box_QP} is
\begin{equation}\label{eqn_initialization_stragegy}
    z^0 = 0,\gamma^0 =1 - \lambda \tilde{h},\theta^0 =1 + \lambda \tilde{h},\phi^0 = e, \psi^0 = e.
\end{equation}
This set of values is in $\mathcal{F}^0$.
\end{remark}

\subsection{Algorithm description}
To simplify the presentation, two vectors $v=\mathrm{col}(\gamma, \theta) \in \mathbb{R}^{2n}$ and $s=\mathrm{col}(\phi,\psi) \in \mathbb{R}^{2n}$ are introduced. equations (\ref{eqn_KKT_tau_d}) and (\ref{eqn_KKT_tau_e}) are replaced by
\[
    v s = \tau^2e
\]
which is then replaced by
\begin{equation}
    \varphi(vs) = \varphi(\tau^2e)\label{eqn_new_complementary}
\end{equation}
where the function $\varphi: \mathbb{R}^n_+ \rightarrow\mathbb{R}^n_+$, is differentiable on $\mathbb{R}^n_{++}$ such that $\varphi(w)>0$ and $\varphi^{\prime}(w)>0$ for all $w>0$. 
\begin{remark}\label{remark_1}
The classical path-following method is recovered for  $\varphi(w)=w$. Here we consider $\varphi(w)=\sqrt{w}$, then $\varphi^{\prime}(w)=\frac{1}{2\sqrt{w}}$.
\end{remark}
Equation \eqref{eqn_new_complementary} is then linearized as 
\begin{equation}\label{eqn_linearized_new_complementary}
s\varphi^{\prime}(v s)\Delta v  +v\varphi^{\prime}(vs)\Delta s=\varphi(\tau^2e)-\varphi(vs).
\end{equation}
Suppose that $(z, v, s)\in \mathcal{F}^0$ and, according to Remark \ref{remark_1}, a direction $(\Delta z,\Delta v,\Delta s)$ can thus be obtained by solving the system of linear equations
\begin{subequations}\label{eqn_newKKT_compact}
    \begin{align}
        &2\lambda\tilde{H}\Delta z + \Omega \Delta v = 0,\label{eqn_newKKT_compact_a}\\
        & \Omega^T \Delta z + \Delta s = 0,\label{eqn_newKKT_compact_b}\\
        &\sqrt{\frac{s}{v}}\Delta v + \sqrt{\frac{v}{s}}\Delta s = 2(\tau e-\sqrt{v s}),\label{eqn_newKKT_compact_c}
    \end{align}
\end{subequations}
where $\Omega=[I,-I] \in\mathbb{R}^{n \times 2n}$. By letting
\begin{subequations}\label{eqn_Delta_gamma_theta_phi_psi}
    \begin{align}
        &\Delta \gamma=\frac{\gamma}{\phi}\Delta z+2\!\left(\sqrt{\frac{\gamma}{\phi}}\tau e-\gamma\right)\!,\\
        &\Delta \theta=-\frac{\theta}{\psi}\Delta z+2\!\left(\sqrt{\frac{\theta}{\psi}}\tau e-\theta\right)\!,\\
        &\Delta\phi = - \Delta z,\\
        &\Delta\psi = \Delta z,
    \end{align}
\end{subequations}
\eqref{eqn_newKKT_compact} can be reduced into a more compacted system of linear equations,
\begin{equation}{\label{eqn_compact_linsys}}
    \begin{aligned}
\Big(2\lambda\tilde{H}+&\diag\!\Big(\frac{\gamma}{\phi}\Big) + \diag\!\Big(\frac{\theta}{\psi}\Big) \!\Big) \Delta z\\
        &=2\!\left(\sqrt{\frac{\theta}{\psi}}\tau e-\sqrt{\frac{\gamma}{\phi}}\tau e+ \gamma - \theta\right).
    \end{aligned}
\end{equation}
The proposed feasible full-Newton path-following interior-point algorithm is summarized on Algorithm \ref{alg_fullNewton}. In the next section, we prove that Algorithm \ref{alg_fullNewton} converges to the optimal solution of Problem \ref{problem_Box_QP} in $\mathcal{N}=\left\lceil\frac{\log(\frac{2n}{\epsilon})}{-2\log(\frac{\sqrt{2n}}{\sqrt{2n}+\sqrt{2}-1})}\right\rceil\!+1$ iterations.

\begin{algorithm}
    \caption{A \textit{direct} optimization algorithm for input-constrained MPC (\ref{problem_input_MPC})
    }\label{alg_fullNewton}
    \textbf{if }$\|h\|_\infty=0$, \textbf{return} $y^*=\frac{1}{2}(u(t)+l(t))$;\\
    \textbf{otherwise},\\
    Let $\lambda=\frac{1}{\sqrt{n+1}}$, cache $2\lambda\Tilde{H}=\frac{2\lambda}{\|h\|_\infty}H$, $\tilde{h}=\frac{1}{\|h\|_\infty}h$. and $(z,\gamma,\theta,\phi,\psi)$ are initialized from (\ref{eqn_initialization_stragegy}), $\eta=\frac{\sqrt{2}-1}{\sqrt{2n}+\sqrt{2}-1}$ and $\tau=\frac{1}{1-\eta}$, and given a stopping tolerance $\epsilon$, thus the required exact number of iterations $\mathcal{N}=\left\lceil\frac{\log(\frac{2n}{\epsilon})}{-2\log(\frac{\sqrt{2n}}{\sqrt{2n}+\sqrt{2}-1})}\right\rceil+1$.
    \\
    \textbf{for} $k=1, 2,\cdots{}, \mathcal{N}$ \textbf{do}
    \begin{enumerate}[label*=\arabic*., ref=\theenumi{}]
        \item $\tau\leftarrow(1-\eta)\tau$;
        \item solve \eqref{eqn_compact_linsys} for $\Delta z$ by using Cholesky decomposition;
        \item calculate $(\Delta\gamma,\Delta\theta,\Delta\phi,\Delta\psi)$ from \eqref{eqn_Delta_gamma_theta_phi_psi};
        \item  $z\leftarrow z+\Delta z$, $\gamma\leftarrow \gamma+\Delta \gamma$, $\theta\leftarrow \theta+\Delta \theta$, $\phi\leftarrow \phi+\Delta \phi$, $\psi\leftarrow \psi+\Delta \psi$;
    \end{enumerate}
    \textbf{end}\\
    \textbf{return $y^*= \frac{1}{2}\diag(u(t)-l(t))z+\frac{1}{2}(u(t)+l(t))$}.
\end{algorithm}

\subsection{Convergence and worst-case analysis}
For simplify the presentation, introduce
\[
d_v:=\sqrt{\frac{s}{v}}\Delta v, \quad d_s:=\sqrt{\frac{v}{s}}\Delta s,
\]
for which $d_v d_s = \Delta v \Delta s$ and $d_v^\top d_s = \Delta v^\top \Delta s$. Equations \eqref{eqn_newKKT_compact_a} and \eqref{eqn_newKKT_compact_b} imply that
\[
\Delta v^{\!\top\!}
\Delta s=\Delta v^{\!\top\!} (-\Omega^{\!\top\!}\Delta z)=(-\Omega\Delta v)^{\!\top\!}\Delta z=\Delta z^{\!\top\!} (2\lambda\tilde{H})\Delta z.
\]
The positive definiteness of $2\lambda\tilde{H}$ implies that $\Delta z^\top (2\lambda\tilde{H}) \Delta z\geq0$ for any vector $\Delta z$, thus $d_v^\top d_s\geq0$. Then introduce
\[
p:=d_v+d_s,\quad q:=d_v-d_s.
\]
Then $(p^2-q^2)/4 = d_v d_s$ and $(\|p\|^2-\|q\|^2)/4=d_v^\top d_s\geq0$, and
\begin{equation}\label{eqn_qv_pv}
    \|q\|\leq\|p\|.
\end{equation}
Now introduce $\beta:=\sqrt{vs}$; then  \eqref{eqn_newKKT_compact_c} implies that
\begin{equation}\label{eqn_p_v}
    p=2(\tau e-\sqrt{vs})=2(\tau e-\beta).
\end{equation}
With the definition of the proximity measure
\begin{equation}\label{eqn_xi}
\xi(\beta,\tau)=\tfrac{\|\tau e-\beta\|}{\tau}=
\tfrac{\|p\|}{2\tau},
\end{equation}
next we prove that, for small enough proximity measure the full-Newton step will not violate the bound \eqref{eqn_KKT_f}. 
That is, the full-Newton step is strictly feasible.
\begin{lemma}\label{lemma_strictly_feasible}
Let $\xi:=\xi(\beta,\tau) < 1$. Then the full-Newton step is strictly feasible, that is, $v_{+}=v+\Delta v>0$ and $s_{+}=s+\Delta s>0$.
\end{lemma}
\begin{proof} For each $0\leq\alpha\leq1$, let $v_{+}(\alpha)=v+\alpha\Delta v$ and $s_{+}(\alpha)=s+\alpha\Delta s$. Then
\[
\begin{aligned}
 v_{+}(\alpha) s_{+}(\alpha)&=v s+\alpha(v \Delta s+s\Delta v)+\alpha^2\Delta v \Delta s\\
 &=v s+\alpha\beta(d_v+d_s)+\alpha^2 d_vd_s\\
    &=\beta^2+\alpha\beta p+\alpha^2\big(\tfrac{p^2}{4}-\tfrac{q^2}{4}\big)\\
    &=(1-\alpha)\beta^2+\alpha(\beta^2+\beta p)+\alpha^2\big(\tfrac{p^2}{4}-\tfrac{q^2}{4}\big)\\
\end{aligned}
\]
From (\ref{eqn_p_v}) we have $\beta + \frac{p}{2}=\tau e$, and $\beta^2+\beta p=\tau^2e-\frac{p^2}{4}$; then 
\begin{equation}\label{eqn_v+_s+}
    v_{+}(\alpha) s_{+}(\alpha)=(1-\alpha)\beta^2+\alpha\big(\tau^2e-\tfrac{p^2}{4}+\alpha(\tfrac{p^2}{4}-\tfrac{q^2}{4})\big)
\end{equation}
Thus, the inequality $v_{+}(\alpha)s_{+}(\alpha)>0$ holds if
\[
\left\|(1-\alpha) \tfrac{p^2}{4}+\alpha \tfrac{q^2}{4}\right\|_{\infty}\!<\tau^2.
\]
Using \eqref{eqn_qv_pv} and \eqref{eqn_xi}, if $\xi<1$, then
\[
\begin{aligned}
\left\|(1-\alpha) \tfrac{p^2}{4}+\alpha \tfrac{q^2}{4}\right\|_{\infty}&\leq(1-\alpha)\left\|\tfrac{p^2}{4}\right\|_{\infty}+\alpha\left\|\tfrac{q^2}{4}\right\|_{\infty}\\
&\leq(1-\alpha)\tfrac{\left\|p\right\|^2}{4}+\alpha\tfrac{\left\|q\right\|^2}{4} \leq\tfrac{\|p\|^2}{4}\\
&=\xi^2\tau^2<\tau^2.
\end{aligned}
\]
Hence, for any $0\leq\alpha\leq1$, we have $v_{+}(\alpha)s_{+}(\alpha)>0$. As a result, the linear functions of $\alpha$, $v_{+}(\alpha)$ and $s_{+}(\alpha)$, do not change sign on the interval $[0,1]$. For $\alpha=0$, we have $v_{+}(0)=v>0$ and $s_{+}(0)=s>0$ thus $v(1)>0$ and $s(1)>0$. This completes the lemma.
\end{proof}

\vspace{0.2cm} 

Next is to prove there exists an upper bound for the duality gap after a full-Newton step.
\begin{lemma}\label{lemma_duality_gap}
    After a full-Newton step, let $v_{+}=v+\Delta v$ and $s_{+}=s+\Delta s$, then the duality gap satisfies 
    \[
        v_{+}^\top s_{+}\leq(2n)\tau^2.
    \]
\end{lemma}
\begin{proof}
    Suppose $\xi<1$ so from Lemma \ref{lemma_strictly_feasible} we obtain that $v_+>0$ and $s_+>0$. Now substituting $\alpha=1$ into \eqref{eqn_v+_s+} gives
    \[
        v_+s_+=\beta_+^2 = \tau^2e-\tfrac{q^2}{4}
    \]
    so we have 
\begin{equation}\label{eqn_beta_+}
    \begin{aligned}
v_{+}^\top s_{+}&=e^\top(v_+s_+)=(2n)\tau^2-\tfrac{e^{\!\top\!} (q^2)}{4}\\
        &=(2n)\tau^2-\tfrac{\|q\|^2}{4}\leq(2n)\tau^2. 
    \end{aligned}
    \end{equation}
    This completes the lemma.
\end{proof}

Thus, the duality gap will converge to the given stopping criteria if $(2n)\tau^2$ converges to the given stopping criteria. In the below lemma, we investigate how the proximity measure $\xi(\beta_+,\tau_+)$ changes after a full-Newton step and an update of $\tau$.
\begin{lemma}\label{lemma_xi}
Suppose that $\xi=\xi(\beta,\tau)<1$ and $\tau_+=(1-\eta)\tau$ where $0<\eta<1$. Then
\[
\xi_+=\xi(\beta_+,\tau_+)\leq\tfrac{\xi^2}{1+\sqrt{1-\xi^2}}+\tfrac{\eta\sqrt{2n}}{1-\eta}.
\]
Furthermore, if $\xi\leq\frac{1}{\sqrt{2}}$ and $\eta=\frac{\sqrt{2}-1}{\sqrt{2n}+\sqrt{2}-1}$ then $\xi_+\leq\frac{1}{\sqrt{2}}$.
\end{lemma}

\begin{proof}
Let $\tau_+=(1-\eta)\tau$, then
\begin{equation}\label{eqn_xi_+_part1}
\begin{aligned}
\xi_+&=\xi(\beta_+,\tau_+)=\frac{\|\tau_+e-\beta_+\|}{\tau_+}\\
&=\frac{\|(1-\eta)\tau e-(1-\eta)\beta_+-\eta\beta_+\|}{(1-\eta)\tau}\\
&\leq\frac{\|\tau e-\beta_+\|}{\tau}+\frac{\eta}{1-\eta}\frac{\|\beta_+\|}{\tau}
\end{aligned}
\end{equation}
Equation \eqref{eqn_beta_+}) implies that
\[
\frac{\|\beta_+\|}{\tau}\leq\sqrt{2n}
\]
and
\[
\begin{aligned}
   \min(\beta_+^2)&=\min\,(\tau^2e-\tfrac{q^2}{4})=\tau^2-\tfrac{\|q^2\|_\infty}{4}\\
   &\geq\tau^2-\tfrac{\|q\|^2}{4}\geq\tau^2-\tfrac{\|p\|^2}{4}\\
   &=\tau^2(1-\xi^2)
\end{aligned}
\]
which yields
\begin{equation}\label{eqn_min_beta_+}
    \min(\beta_+)\geq\tau\sqrt{1-\xi^2}.
\end{equation}
Furthermore, from (\ref{eqn_qv_pv}), (\ref{eqn_beta_+}), (\ref{eqn_min_beta_+}) and the Cauchy–Schwarz inequality,
\[
\begin{aligned}
 \frac{\|\tau e-\beta_{+}\|}{\tau}&=\frac{1}{\tau}\left\|\frac{\tau^2 e-\beta_+^2}{\tau e+\beta_+}\right\|\leq\frac{1}{\tau}\frac{\|\tau^2e-\beta_+^2\|}{\min(\tau e+\beta_+)}\\
 &=\frac{1}{\tau}\frac{\|\tau^2e-\beta_+^2\|}{\tau+\min(\beta_+)}\leq\frac{\|\tau^2 e-\beta_+^2\|}{\tau^2(1+\sqrt{1-\xi^2})}\\
 &=\frac{\|q^2\|}{4\tau^2(1+\sqrt{1-\xi^2})}\leq\frac{\|q\|^2}{4\tau^2(1+\sqrt{1-\xi^2})}\\
 &\leq\frac{\|p\|^2}{4\tau^2(1+\sqrt{1-\xi^2})}=\frac{\xi^2}{1+\sqrt{1-\xi^2}}
\end{aligned}
\]
thus, based on \eqref{eqn_xi_+_part1}, we have
\[
\xi_+=\xi(\beta_+,\tau_+)\leq\frac{\xi^2}{1+\sqrt{1-\xi^2}}+\frac{\eta\sqrt{2n}}{1-\eta}
\]
This proves the first part of the lemma. 
Now let $\eta=\frac{\sqrt{2}-1}{\sqrt{2n}+\sqrt{2}-1}$, and if $\xi\leq\frac{1}{\sqrt{2}}$, we deduce $\frac{\xi^2}{1+\sqrt{1-\xi^2}}\leq\frac{2-\sqrt{2}}{2}$. Thus,
\[
\begin{aligned}
  \xi_+&\leq\frac{2-\sqrt{2}}{2}+\frac{\frac{\sqrt{2}-1}{\sqrt{2n}+\sqrt{2}-1}}{1-\frac{\sqrt{2}-1}{\sqrt{2n}+\sqrt{2}-1}}\sqrt{2n}=\frac{1}{\sqrt{2}}.
\end{aligned}
\]
The proof of the lemma is complete.
\end{proof}

\begin{lemma}\label{lemma_xi_condition}
The value of $\xi(\beta,\tau)$ before the first iteration is denoted as
$\xi^0=\xi(\beta^0,(1-\eta)\tau^0)$. If $(1-\eta)\tau^0=1$, $\lambda=\frac{1}{\sqrt{n+1}}$, then $\xi^0\leq\frac{1}{\sqrt{2}}$ and $\xi(\beta, w)\leq\frac{1}{\sqrt{2}}$ is always satisfied.
\end{lemma}
\begin{proof}
The equality $(1-\eta)\tau^0=1$ implies that
\[
\begin{aligned}
\xi^0&=\frac{\|(1-\eta)\tau^0e-\beta^0\|}{(1-\eta)\tau^0}=
\|e-\beta^0\|\\
&=\sqrt{\sum_{i=1}^n\left(1-\sqrt{1-\lambda \tilde{h}_i}\right)^{\!\!2}+\left(1-\sqrt{1+\lambda \tilde{h}_i}\right)^{\!\!2}}\\
&=\sqrt{2n+2n-2\sum_{i=1}^{n}\sqrt{1-\lambda \tilde{h}_i}+\sqrt{1+\lambda \tilde{h}_i}}
\end{aligned}
\]
Denote $m_i=\sqrt{1-\lambda \tilde{h}_i}+\sqrt{1+\lambda \tilde{h}_i}$; then 
\[
m_i^2=1-\lambda\tilde{h}_i+1+\lambda\tilde{h}_i+2\sqrt{1-\lambda^2\tilde{h}_i^2}=2+2\sqrt{1-\lambda^2\tilde{h}_i^2}
\]
Since $\|\tilde{h}\|_\infty=1$, $m_i^2\geq2+2\sqrt{1-\lambda^2}$, that is, $m_i\geq\sqrt{2+2\sqrt{1-\lambda^2}}$. Also implied is
\[
\xi^0=\sqrt{4n-2\sum_{i=1}^{n}m_i}\leq\sqrt{4n-2n\sqrt{2+2\sqrt{1-\lambda^2}}}
\]
thus the inequality $\xi^0\leq\frac{1}{\sqrt{2}}$ holds if
\[
\begin{aligned}
 &4n-2n\sqrt{2+2\sqrt{1-\lambda^2}}\leq\!\tfrac{1}{2}\\\
 \leftrightarrows\quad&2-\tfrac{1}{4n}\leq\sqrt{2+2\sqrt{1-\lambda^2}}\\
\leftrightarrows\quad&2+2\sqrt{1-\lambda^2}\geq\big(2-\tfrac{1}{4n}\big)^2\\
\leftrightarrows\quad&\sqrt{1-\lambda^2}\geq1-\tfrac{1}{2n}+\tfrac{1}{32n^2}\\
\leftrightarrows\quad&\lambda^2\leq1-\left(1-\tfrac{1}{2n}+\tfrac{1}{32n^2}\right)^{2}\\
&=\tfrac{1}{n}-\tfrac{5}{16n^2}+\tfrac{1}{32n^3}-\tfrac{1}{1024n^4}
\end{aligned}
\]
For $\lambda=\frac{1}{\sqrt{n+1}}$; this inequality holds if
\[
\begin{aligned}
 &\tfrac{1}{n+1}\leq\tfrac{1}{n}-\tfrac{5}{16n^2}+\tfrac{1}{32n^3}-\tfrac{1}{1024n^4}, \ \forall n\in \mathbb{N}_+\\
 \leftrightarrows\quad&\tfrac{1}{n}\big(\tfrac{1}{n+1}-\tfrac{5}{16n}+\tfrac{1}{32n^2}-\tfrac{1}{1024n^3}\big)\!\geq0, \ \forall n\in \mathbb{N}_+\\
 \leftrightarrows\quad&\tfrac{n^3}{n+1}\geq\tfrac{5}{16}n^2-\tfrac{1}{32}n+\tfrac{1}{1024}, \ \forall n\in \mathbb{N}_+\\
\leftrightarrows\quad&n^3\geq\left(\tfrac{5}{16}n^2-\tfrac{1}{32}n+\tfrac{1}{1024}\right)\left(n+1\right),\ \forall n\in \mathbb{N}_+\\
\leftrightarrows\quad&n^3\geq\tfrac{5}{16}n^3+\tfrac{9}{32}n^2-\tfrac{31}{1024}n+\tfrac{1}{1024}, \ \forall n\in \mathbb{N}_+\\
\leftrightarrows\quad&\tfrac{11}{16}n^3-\tfrac{9}{32}n^2+\tfrac{31}{1024}n-\tfrac{1}{1024}\geq0, \ \forall n\in \mathbb{N}_+\\
\leftrightarrows\quad&\tfrac{n^2}{32}(22n-9)+\tfrac{1}{1024}(31n-1)\geq0,\ \forall n\in \mathbb{N}_+
\end{aligned}
\]
Obviously, the last inequality holds, thus the first part of the lemma is proved.

From Lemma \ref{lemma_xi} if $\xi^0\leq\frac{1}{\sqrt{2}}$ is satisfied then $\xi(\beta,\tau)\leq\frac{1}{\sqrt{2}}$ is always satisfied through the iterations. The proof of the lemma is complete.
\end{proof}
\begin{remark}
    From Lemma \ref{lemma_xi_condition}, the assumption of Lemmas \ref{lemma_strictly_feasible}
    and \ref{lemma_xi}, namely $\xi(\beta,\tau)<1$, is satisfied.
\end{remark}
\begin{lemma}\label{lemma_max_iters}
Given $v^0=\mathrm{col}(\gamma^0,\theta^0)$ and $s^0=\mathrm{col}(\phi^0,\psi^0)$ from (\ref{eqn_initialization_stragegy}), they are strictly feasible. Let $v_k,s_k$ be the $k$th iterates of $v,s$, then the inequalities $v_k^{\top\!} s_k\leq \epsilon$ is satisfied for 
\begin{equation}
    k \geq\! \left\lceil \frac{\log(\frac{2n(\tau^0)^2}{\epsilon})}{-2\log(1-\eta)}\right\rceil.
\end{equation}
\end{lemma}
\begin{proof}
Let $\tau_k$ be the $k$th iterate of $\tau$, so $\tau_k=(1-\eta)^k\tau^0$. Applying Lemma \ref{lemma_duality_gap} gives that
\[
v_k^\top s^k\leq2n\tau_k^2=2n(1-\eta)^{2k}(\tau^0)^2.
\]
Hence $v_k^\top s_k\leq\epsilon$ holds if 
\[
2n(1-\eta)^{2k}(\tau^0)^2\leq\epsilon.
\]
Taking logarithms gives that
\begin{equation}\label{eqn_log}
 2k \log(1-\eta)+\log(2n(\tau^0)^2)\leq\log\epsilon, 
\end{equation}
which holds if
\[
k \geq \!\left\lceil \frac{\log(\frac{2n(\tau^0)^2}{\epsilon})}{-2\log(1-\eta)}\right\rceil.
\]
The proof is complete.
\end{proof}

\begin{theorem}\label{theorem_at_most}
Let $\eta=\frac{\sqrt{2}-1}{\sqrt{2n}+\sqrt{2}-1}$ and $\tau^0 =\frac{1}{1-\eta}$, Algorithm \ref{alg_fullNewton} requires at most
\begin{equation}\label{eqn_worst_number_iteragions}
N_{\operatorname{\max}}=\!\left\lceil\frac{\log(\frac{2n}{\epsilon})}{-2\log(\frac{\sqrt{2n}}{\sqrt{2n}+\sqrt{2}-1})}\right\rceil\! + 1
\end{equation}
iterations, which gives that $v^\top s\leq\epsilon$.
\end{theorem}
\begin{proof}
By Lemmas \ref{lemma_duality_gap}--
\ref{lemma_max_iters}, let $\eta=\frac{\sqrt{2}-1}{\sqrt{2n}+\sqrt{2}-1}$, and $\tau^0=\frac{1}{1-\eta}$ to satisfy $(1-\eta)\tau^0=1$. Thus Algorithm \ref{alg_fullNewton} requires at most
\[
\begin{aligned}
N_{\operatorname{\max}}&=\!\left\lceil
\frac{\log(\frac{2n}{\epsilon})}{-2\log(1-\eta)}+ \frac{2\log(\tau^0)}{-2\log(1-\eta)}
\right\rceil\\
&=\!\left\lceil
\frac{\log(\frac{2n}{\epsilon})}{-2\log(1-\eta)}
\right\rceil\! + 1\\
&=\!\left\lceil
\frac{\log(\frac{2n}{\epsilon})}{-2\log(\frac{\sqrt{2n}}{\sqrt{2n}+\sqrt{2}-1})}
\right\rceil \! + 1
\end{aligned}
\]
iterations. The proof is complete.
\end{proof}

Unlike other methods whose iteration complexity analysis is conservative (its actual number of iterations is smaller than the maximum number of iterations, resulting in oversized control computing estimates), the worst-case computation analysis of our proposed Algorithm \ref{alg_fullNewton} is \textit{exact} and deterministic.
\subsection{Worst-case analysis is deterministic}
The worst-case iteration complexity analysis is based on the relationship between the duality gap $v^\top s$ and $(2n)\tau^2$, as shown in \eqref{eqn_beta_+}. In fact, the two are nearly equal to each other in the proposed algorithm framework, which indicates that the worst-case iteration analysis is deterministic with no conservativeness.
\begin{theorem}\label{theorem_exact}
    Let $\eta=\frac{\sqrt{2}-1}{\sqrt{2n}+\sqrt{2}-1}$ and $\tau^0=\frac{1}{1-\eta}$, Algorithm \ref{alg_fullNewton} exactly requires
    \begin{equation}
    \mathcal{N}=\!\left\lceil\frac{\log(\frac{2n}{\epsilon})}{-2\log(\frac{\sqrt{2n}}{\sqrt{2n}+\sqrt{2}-1})}\right\rceil\!+ 1
    \end{equation} 
    iterations, the resulting vectors have $v^\top s\leq\epsilon$.
\end{theorem}
\begin{proof}
Equations \eqref{eqn_xi} and \eqref{eqn_beta_+}) and $\xi\leq\frac{1}{\sqrt{2}}$ give that
\[
0\leq\tfrac{\|q\|^2}{4}\leq\tfrac{\|p\|^2}{4}=\xi^2\tau^2=\tfrac{\xi^2}{2n}(2n\tau^2)\leq\tfrac{1}{4n}(2n\tau^2),
\]
that is,
\[
(1-\tfrac{1}{4n})2n\tau_k^2\leq v_k^\top s_k\leq2n\tau_k^2.
\]
The larger the value of $n$, the tighter the bounds on $v_k^\top s_k$. The duality gap $v_k^\top s_k$ has nearly equal decreasing behavior with $\tau_k$; the number of iterations would be exact instead of worst-case. 
To prove it, consider the $(k-1)$th iteration, 
\[
(1-\tfrac{1}{4n})2n\tau_{k-1}^2\leq v_{k-1}^\top s_{k-1}\leq2n\tau_{k-1}^2.
\]
If $2n\tau_{k}^2\leq\left(1-\frac{1}{4n}\right)2n\tau_{k-1}^2$, then the duality gap will reach the convergence criterion no earlier than and no later than $2n\tau_k^2$. That is, we need to prove that
\[
(1-\eta)^2\leq1-\tfrac{1}{4n},\ \forall n=1,2,\cdots
\]
Substituting $\eta=\frac{\sqrt{2}-1}{\sqrt{2n}+\sqrt{2}-1}$ into the above, we need to prove that
\[
\tfrac{(\sqrt{2}-1)^2}{(\sqrt{2n}+\sqrt{2}-1)^2}-\tfrac{2(\sqrt{2}-1)}{\sqrt{2n}+\sqrt{2}-1}+\tfrac{1}{4n}\leq0, \ \forall n=1,2,\cdots
\]
which is equal to
\[
\begin{aligned}
    -&(4-2\sqrt{2})\sqrt{n}(n-1)-(8\sqrt{2}-10)n(\sqrt{n}-1)\\
    -&(3-2\sqrt{2})(n^{3/2}-1)-(19-12\sqrt{2})n^{3/2}\leq0
\end{aligned}
\]
which obviously always holds for $n=1,2,\cdots$.

From (\ref{eqn_log}), let $\eta=\frac{\sqrt{2}-1}{\sqrt{2n}+\sqrt{2}-1}$ and $\tau^0=\frac{1}{1-\eta}$; the required exact number of iterations is
\[
\mathcal{N}=\!\left\lceil\frac{\log(\frac{2n}{\epsilon})}{-2\log(\frac{\sqrt{2n}}{\sqrt{2n}+\sqrt{2}-1})}\right\rceil \! + 1
\]
The proof is complete.
\end{proof}

\subsection{Execution-time certificate}
Since our proposed Algorithm \ref{alg_fullNewton} is a full-Newton IPM algorithm without a line search procedure, each step of Algorithm \ref{alg_fullNewton} involves a clear and countable number of floating-point operations ([flops]). Our proposed Algorithm \ref{alg_fullNewton} has an exact number of iterations which is \textit{only dimension-dependent} (\textit{data-independent}), thus we can summarize the total [flops] required by Algorithm \ref{alg_fullNewton} as follows.
\begin{theorem}\label{theorem}
In Algorithm \ref{alg_fullNewton}, the initialization requires $(n) + (3) + (1+n^2) + (n) + (5n)+ (5)+ (2)$ [flops], Step 1 requires $2$ [flops], Step 2 requires $(\frac{1}{3}n^3+\frac{1}{2}n^2+\frac{1}{6}n)+2n^2+11n$ [flops], Step 3 requires $6n$ [flops], Step 4 requires $5n$ [flops]. Thus, Algorithm \ref{alg_fullNewton} totally requires $n^2+7n+11+\mathcal{N}(\frac{1}{3}n^3+\frac{5}{2}n^2+\frac{133}{6}n+2)$ [flops].
\end{theorem}

Then, based on the assumption that \textit{the adopted computation platform performs a fixed [flops] in constant time},
$$
\textrm{execution time} = \frac{\textrm{total [flops] required by the algorithm}}{\textrm{average [flops] processed per second}}~[s].
$$
Thus, by Theorem 3, Algorithm 1 can provide an execution-time certificate, which is \textit{only dimension-dependent} (\textit{data-independent}) making it competent for the execution-time certification of time-varying MPC problems (with possible time-varying costs or dynamics, such as RTI-based nonlinear MPC \cite{gros2020linear}).

\section{Numerical example}
In this section, Algorithm~\ref{alg_fullNewton} is implemented in MATLAB2023a via a C-mex interface, and the closed-loop simulation is performed on a contemporary MacBook Pro with 2.7~GHz 4-core Intel Core i7 processors and 16GB RAM. We test Algorithm \ref{alg_fullNewton} on an open-loop unstable AFTI-16 aircraft application. The aim is to validate whether the practical number of iterations is the same with the theoretical $\!\left\lceil\frac{\log(\frac{2n}{\epsilon})}{-2\log(\frac{\sqrt{2n}}{\sqrt{2n}+\sqrt{2}-1})}\right\rceil \!+ 1$, and whether the certified execution-time is smaller than the adopted sampling time. The open-loop unstable linearized AFTI-16 aircraft model reported in \cite{bemporad1997nonlinear} 
is 
\[
\left\{\begin{aligned}
\dot{x} =&{\footnotesize\left[\begin{array}{cccc}
-0.0151 & -60.5651 & 0 & -32.174 \\
-0.0001 & -1.3411 & 0.9929 & 0 \\
0.00018 & 43.2541 & -0.86939 & 0 \\
0 & 0 & 1 & 0
\end{array}\right]} x\\&+\!{\footnotesize\left[\begin{array}{cc}
-2.516 & -13.136 \\
-0.1689 & -0.2514 \\
-17.251 & -1.5766 \\
0 & 0
\end{array}\right] }u \\
y =&{\footnotesize\left[\begin{array}{llll}
0 & 1 & 0 & 0 \\
0 & 0 & 0 & 1
\end{array}\right]x}
\end{aligned}\right.
\]
We choose the sampling time as $\Delta t=0.05$~s, then the model is sampled using zero-order hold every $\Delta t=0.05$~s. The input constraints are $|u_i| \leq 25^{\circ},i = 1, 2$. The control goal is to make the pitch angle $y_2$ track a reference signal $r_2$. The cost matrices $W_y = \diag$([10,10]), $W_u = 0$, and $W_{\Delta u}= \diag$([0.1, 0.1]) are used in the MPC design. We investigate our proposed Algorithm \ref{alg_fullNewton} among different prediction horizon settings, $T=5,10,15,20$, which results in different problem dimensions $n=10,20,30,40$ as the dimension of control inputs is 2. We adopt the stopping convergence criteria $\epsilon=10^{-6}$. By Theorem \ref{theorem_exact}, the derived theoretical number of iterations is 
$\!\left\lceil\frac{\log(\frac{2n}{\epsilon})}{-2\log(\frac{\sqrt{2n}}{\sqrt{2n}+\sqrt{2}-1})}\right\rceil \!+ 1$, that is,
$
[96,\quad139,\quad173,\quad202] 
$
for different prediction horizon settings. By Theorem \ref{theorem}, before closed-loop simulations, we can exactly calculate the total [flops] of different prediction horizon settings, see Table \ref{tab1}.

In Table \ref{tab1}, the derived execution time in \textit{theory} that can be regarded as an execution-time certificate, is assumed to perform on 1 Gflop/s computing processor, and the execution time in \textit{practice} is obtained from a contemporary MacBook Pro with 2.7~GHz 4-core Intel Core i7 processors and 16GB RAM. Since 2.7~GHz 4-core Intel Core i7 processor cannot put all its computing power for this calculation (occupied by other PC's tasks), its execution time is still faster than a 1 Gflop/s computing processor from Table \ref{tab1}. The execution time in \textit{theory} (on 1 Gflop/s computing processor) among different prediction horizon settings are all smaller than the sampling time $\Delta t=50$ ms, which are execution-time certificates.

\begin{table}[!htbp]
\caption{Theoretical and practical computation performance of different prediction horizon $(T=5,10,15,20)$ settings}
\centering
\begin{tabular}{crlcrl}
\toprule
Problem  & \multicolumn{2}{c}{Number of iterations} &  Total& \multicolumn{2}{c}{Execution time {(ms)}}\\
dimension & \textit{theory}&\textit{practice} & [flops]$\times 10^{6}$  & \textit{theory}*&\textit{practice}
\\\midrule
 $n=10$ & 96&96 & 0.0777&0.0777 & 0.061\\
 $n=20$ & 139&139 & 0.5721&0.5721 & 0.450\\
 $n=30$ & 173&173 & 2.0628&2.0628 & 1.650\\
 $n=40$ & 202&202 & 5.9287&5.9287 & 4.210\\
\bottomrule
\end{tabular}
\begin{tablenotes}
    \footnotesize
    \item * The theoretical execution time is obtained by assuming running on 1 flop/s computing processor.
\end{tablenotes}
\label{tab1}
\end{table}

All closed-loop simulation results among different prediction horizon settings share almost the same closed-loop performance, as shown in Fig.\ \ref{Fig1}. The outputs $y_1$ and $y_2$ track the reference well while
the inputs $u_1$ and $u_2$ never go beyond $[-25,25]$. 
\begin{figure}[!ht]
        \hspace*{-1em}\includegraphics[width=1.1\columnwidth]{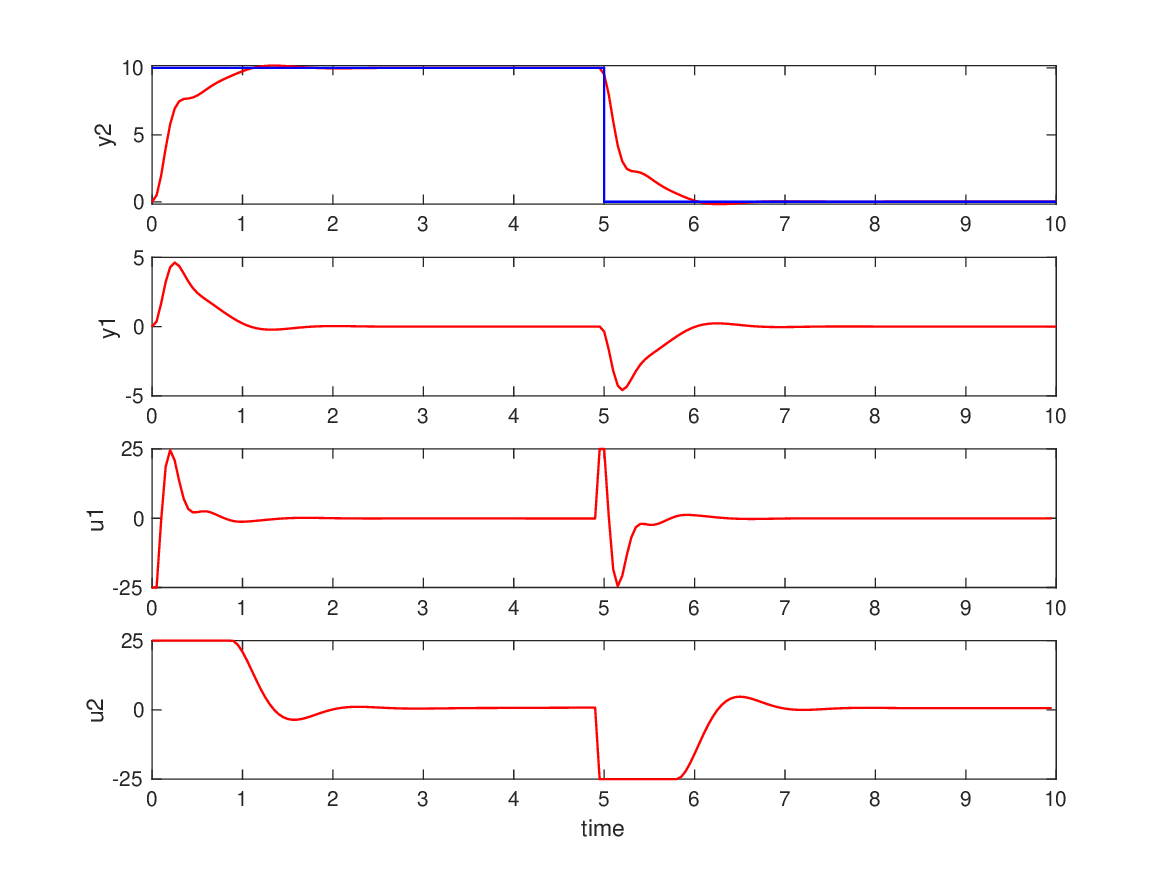} 
        \vspace{-1cm}
        \caption{Closed-loop performance of input-constrained MPC for AFTI-16 among different prediction horizons $(T=5,10,15,20)$.}
        \label{Fig1} 
\end{figure}

\section{Conclusion}
This technical note presents a \textit{direct} optimization algorithm with \textit{only dimension-dependent} (\textit{data-independent}), \textit{simple-calculated} and \textit{exact} 
\[\!\left\lceil\frac{\log(\frac{2n}{\epsilon})}{-2\log(\frac{\sqrt{2n}}{\sqrt{2n}+\sqrt{2}-1})}\right\rceil\!+ 1\] 
number of iterations for certifying the execution time of input-constrained linear MPC problems in real-time closed-loop. The computation complexity of our \textit{direct} optimization algorithm is \textit{only dimension-dependent} and \textit{simple-calculated} features, making it trivially certifying the execution time of nonlinear MPC problems via Koopman operator or RTI scheme, see our recent paper \cite{wu2024time} and \cite{wu2024execution}, respectively. This capability sets our algorithm apart from previous algorithms. One may argue that our algorithm is of very limited use since it targets input-constrained MPC problems. To dispel this concern, our recent paper \cite{wu2024execution_l1} extends our algorithm to encompass general MPC problems with input and state constraints.

Future endeavors will focus on speeding up our execution-time-certified algorithm (with smaller certified execution-time) and exploring practical applications in fast dynamical systems, such as robotics and electronics.

\section{Acknowledgement}
This research was supported by the U.S. Food and Drug Administration under the FDA BAA-22-00123 program, Award Number 75F40122C00200.

\bibliographystyle{IEEEtran}
\bibliography{ref} 
\end{document}